\documentclass{amsart}
\usepackage{graphicx,color}

\theoremstyle{plain}
\newtheorem{proposition}{Proposition}

\theoremstyle{definition}
\newtheorem{definition}{Definition}
\newtheorem{example}{Example}
\newtheorem{notation}{Notation}
\newtheorem{proof 1}{Proof of Proposition}

\newtheorem{remark}{Remark}

\newcommand{\R}{\ensuremath{\mathbb{R}}}
\renewcommand{\H}{\ensuremath{\mathbb{H}}}
\newcommand{\Z}{\ensuremath{\mathbb{Z}}}
\newcommand{\N}{\ensuremath{\mathbb{N}}}

\newcommand{\G}{\ensuremath{\mathbb{G}}}
\newcommand{\binf}{\partial_\infty}

\begin{document}
\title{The visual boundary of $\Z^2$}

\author{Kyle Kitzmiller}
\address{Kyle Kitzmiller; Department of Mechanical Engineering; San Diego State University; 5500 Campanile Drive; San Diego, CA 92182-1323} 
\email{kitz2103@gmail.com}

\author{Matt Rathbun}
\address{Matt Rathbun; Department of Mathematics; University of California, Davis; One Shields Ave.; Davis, 95616}
\email{mrathbun@math.ucdavis.edu}


\maketitle

\begin{abstract}
We introduce ideas from geometric group theory related to boundaries of groups.
We consider the visual boundary of a free abelian group, and show that it is an uncountable set with the trivial topology.
\end{abstract}

\section{Introduction}

The study of a metric space can often be facilitated by considering it in the large scale, or by
studying asymptotic phenomena.  For instance, adding a boundary to compactify (or, more generally, ``bordify") a metric space is a key tool
in understanding the space and its isometry group. A classical example is the hyperbolic space $\H^n$, with its boundary sphere at infinity. Isometries of  $\H^n$ extend to homeomorphisms of the boundary, and can be classified by their fixed points on the boundary. 
More generally, any Gromov hyperbolic space (that is, a space with large-scale negative curvature) has such a naturally defined boundary at infinity.

In trying to understand the geometry of groups, it is often useful to regard the group as a metric space by choosing a generating set, 
forming the associated {\em Cayley graph}, which will be defined below.
The metric induced by declaring all edges in the Cayley graph to have length one is called the {\em word metric} on the group.  
It would seem quite natural to define a  boundary for groups directly from the word metric, and this works well if the group is
Gromov hyperbolic.
In general, however, there are obstructions to the usefulness of this boundary, as we will see below.
This note explores properties of the visual boundary for groups, introducing the needed definitions along the way. 
The main result is that the visual boundary of $\Z^2$ (denoted $\binf(\Z^2)$) with the standard generating set possesses the trivial topology on an uncountable set. 
Indeed, there are many groups which have so called ``quasi-flats," or quasi-isometric embeddings of $\Z^2$.   We will see that the boundary of any such group will inherit the unpleasant properties of $\binf(\Z^2)$.

The exposition is intended to be readable for a student who has had a first course in topology and metric spaces, and who is familiar
with the definition and the most basic examples of groups.  (We also mention the axiom of choice.)  On the other hand, we hope 
that the paper will be a non-trivial read for working mathematicians in other areas.

\vspace{3 mm}
{\bf Acknowledgments.} We would like to thank Moon Duchin for the inspiration of this project, as well as the for all of her motivation and support. The authors were supported by VIGRE NSF grant number 0636297.

\section{Background}

\subsection{Metric notions}
We review here some useful definitions from metric geometry.

\begin{definition} \
\begin{itemize}
\item A \emph{geodesic segment}, {\em ray}, or {\em line}  in a metric space $X$ is an isometric embedding of 
$[0,a]$, $[0,\infty)$, or $\R$ into $X$.
That is, for instance, a geodesic line is a map $f: \R \to X$  such that for all $t_1, t_2 \in \R \, $, 
$d_X(f(t_1), f(t_2)) = |t_1 - t_2|$. We say a geodesic ray is \emph{from} $x_0$ or {\em based at} $x_0$ if $f(0) = x_0$.
\item A metric space is called a \emph{geodesic space} if any two points in the space can be joined by a geodesic segment.
\item Suppose $(X,d)$ is a metric space, and $Y \subset X$ is connected. There are two natural ways to metrize $Y$. The \emph{subspace metric} is $d_Y: Y \to \R_{\geq 0}$ defined by $d_Y(y_1, y_2) = d(y_1, y_2)$. Alternatively, the \emph{path metric} is $d_{path}: Y \to \R_{\geq 0}$ defined by $d_{path}(y_1, y_2) = \inf \{ length(\gamma) \, | \, \gamma \mbox{ is a path in } Y \mbox{ connecting } y_1 \mbox{ to } y_2 \}$. 
\item A geodesic space is called \emph{(geodesically) complete} if every geodesic segment can be extended infinitely in both directions.
\item A metric space is called \emph{proper} if closed balls are compact.  (This is needed for certain kinds of limiting arguments.)
\end{itemize}
\end{definition}

\begin{example} \cite{P} Consider the set $R = \{ (x, y, z) \, | \, x = 1, y = 0, z \geq 0 \} \cup \{ (x, y, z) \, | \, y = 0, z = 5x - 5, 0 \leq x \leq 1 \}$, rotated about the $z$-axis (see Figure \ref{figure:pencil}). Call the resulting set $X$, and give it the path metric as a subset of $\R^3$. Then $X$ is not geodesically complete. Take a geodesic segment from $(\epsilon, 0, 5\epsilon -5)$ to $(0, 0, -5)$ for some small $\epsilon > 0$. Trying to extend this geodesic to $(-\epsilon, 0, 5\epsilon -5)$ presents a problem. The length of the two segments would be $2\sqrt{26}\epsilon$, whereas, the distance between the two points $(\epsilon, 0, 5\epsilon - 5)$ and $(-\epsilon, 0, 5\epsilon - 5)$ is only $\pi \epsilon$ (along the horizontal circle $\{z = 5\epsilon -5\} \cap X$). Certainly, if we try to extend the geodesic to any other point on $X$, we will face the same difficulty: that there is a shorter path ``around" the cone, rather than going through the cone point.
\end{example}

\begin{figure}[tb]

\begin{center}
\includegraphics[width=2in]{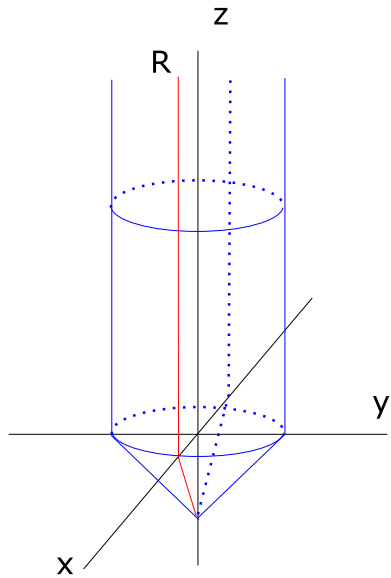}
\caption{The set $R$ rotated around the $z$-axis\label{figure:pencil}
}
\end{center}
\end{figure}

\begin{example} Let $X$ be an infinite-dimensional Hilbert space. Then $X$ is not proper, because the closed unit ball is not compact. To see this, take an orthonormal basis, $\{v_{\alpha} \}$ for $X$. Then any countably infinite sequence of the $\{ v_{\alpha} \}$ is a sequence with no convergent subsequence, since the distance between any two elements is $||v_{\alpha} - v_{\beta}|| = \sqrt{ \langle v_{\alpha}, v_{\alpha} \rangle + \langle v_{\beta}, v_{\beta}\rangle - \langle v_{\alpha}, v_{\beta} \rangle - \langle v_{\beta}, v_{\alpha} \rangle } = \sqrt{ ||v_{\alpha}|| + ||v_{\beta}||} = \sqrt{2}$.
\end{example}

\subsection{Cayley graphs}
The construction of a Cayley graph is a central tool in geometric group theory, allowing us
 to associate a metric space to a group with a given presentation.

\begin{definition} Let $G = \left< S \,  | \, R \right> $ be a group with generating set $S$ and relations $R$. We define a graph $Cay(G, S)$ whose vertices correspond to elements of $G$, and with edges between $g, h \in G$ if there exists $s \in S \cup S^{-1}$ so that $g = h\cdot s$. We give the resulting graph the graph metric, whereby each edge has length $1$, and the distance between vertices is the length of the shortest path between them. 
\end{definition}

\begin{remark} For any two elements $g, h \in G$, the distance from $g$ to $h$ in $Cay(G, S)$ is just the shortest length word in $S \cup S^{-1}$ such that $g = h\cdot s$. 
\end{remark}

\begin{figure}[tb]

\begin{center}
\includegraphics[width=1.5in]{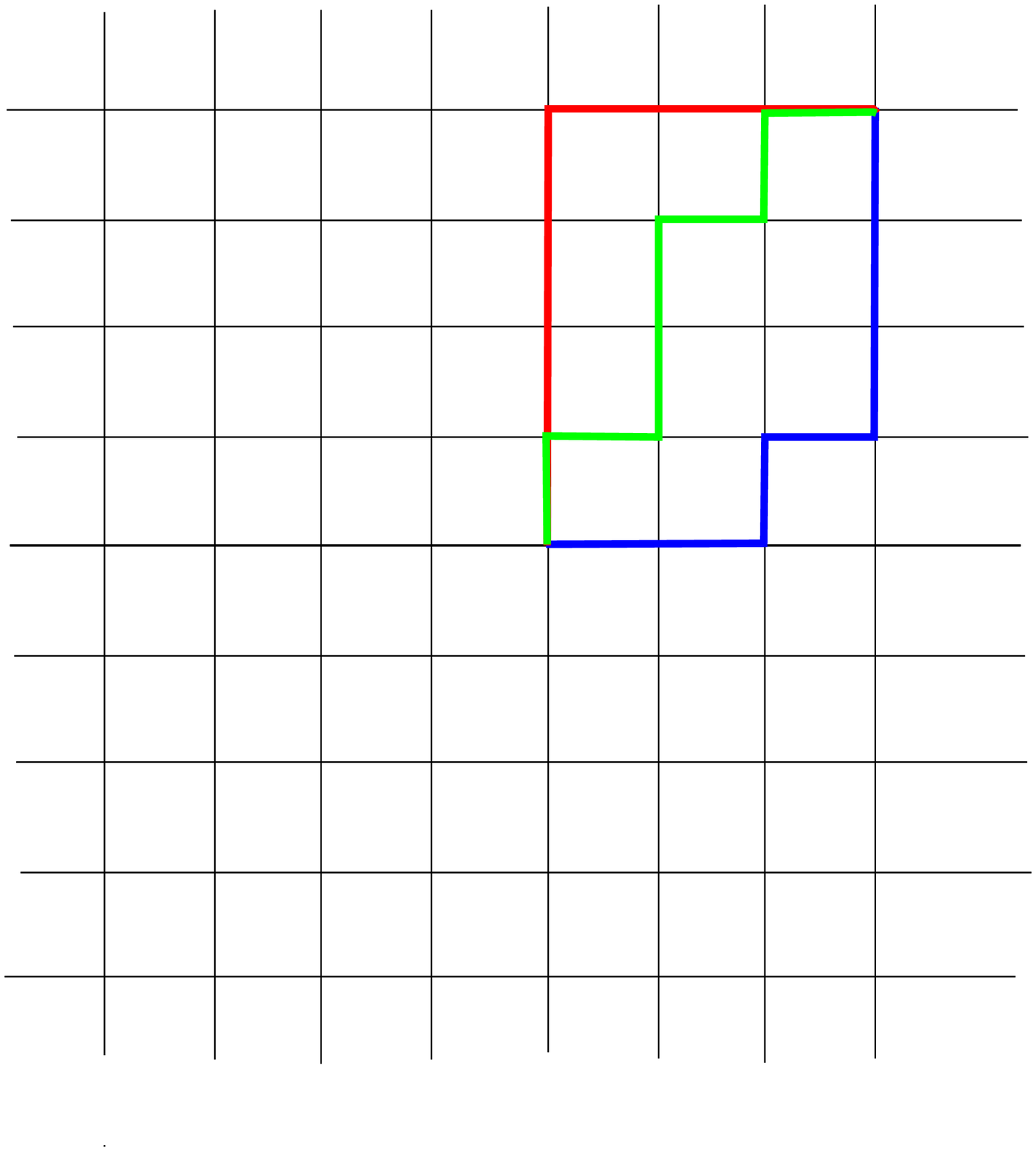}
\caption{Paths in $\Z^2$
\label{figure:Z2}
}
\end{center}
\end{figure}

\begin{example} 
\label{example:Z2}
$Cay(\Z^2, \{(1, 0), (0,1)\})$ is just the integer grid (see figure \ref{figure:Z2}). Consider a path from the origin to any distinct point $(m, n) \in \Z^2$. This path consists of a union of horizontal and vertical segments between the integer coordinate points of the graph, the vertices. There are some crucial differences from familiar metric spaces like $\R^2$ with the Euclidean metric:  there is more than one path of minimum length between the origin and $(m ,n)$ unless $m=0$ or $n=0$, and there is no unique prolongation of geodesic segments to rays.

The distance from $(m, n)$ to $(k, l)$ is $|m - k| + |n - l|$ (the $\ell^1$ distance). Notice that $(m, n) = (k, l) \pm |m - k| (1, 0) \pm |n - l|(0, 1)$, so the distance is the length of the smallest word $s$ composed of letters from $\{\pm(0,1), \pm(1,0)\}$ such that $(m, n) = (k, l) + s$.

Alternatively, one could consider embedding the integer grid into $\R^2$, and take the metric on $\Z^2$ to be the path metric induced by this inclusion.

\end{example}

\begin{remark} Notice that this graph is not determined by a group, but clearly depends on the choice of generating set $S$. To accommodate this, in the next section we introduce the notion of quasi-isometry.
\end{remark}

\subsection{Quasi-isometries}
Often, we want to say that two spaces share some of the same large-scale geometric features, even when they are not isometric. To this end, we introduce the concept of \emph{quasi-isometry}. This is like isometry, but allows for some bounded error in the form of a multiplicative and an additive factor. We will find that many notions about metric spaces can be ``quasified."

\begin{definition}\
\begin{itemize}
\item We say a map between two metric spaces, $f: (X, d_X) \to (Y, d_Y)$ is a \emph{quasi-isometric embedding} for some $k \geq 1$, $c \geq 0$, if for every $x_1, x_2 \in X$,
$$\frac{1}{k}d_X(x_1, x_2) - c \leq d_Y(f(x_1), f(x_2)) \leq kd_X(x_1, x_2) + c.$$

\item We say that a quasi-isometric embedding, $f: (X, d_X) \to (Y, d_Y)$, is a \emph{quasi-surjection} if there exists a $D > 0$ such that for every $y \in Y$, there is an $x \in X$ such that $d_Y(y, f(x)) < D$.

If $f: (X, d_X) \to (Y, d_Y)$ is a quasi-isometric embedding which is also a quasi-surjection, then we say that $f$ is a \emph{quasi-isometry} and we say that $(X, d_X)$ and $(Y, d_Y)$ are \emph{quasi-isometric}.

In particular, a quasi-isometry admits a quasi-inverse. When we compose a quasi-isometry with a quasi-inverse, we almost get the identity. But, as with most things ``quasi", we might be off by a multiplicative and additive constant.

\item If $f: (X, d_X) \to (Y, d_Y)$ is a quasi-isometry, a \emph{quasi-inverse} is a quasi-isometric embedding $g: (Y, d_Y) \to (X, d_X)$ so that for some $k \geq 1$,  $c \geq 0$, for all $x_1, x_2 \in X$, 
$$\frac{1}{k}d_X(x_1, x_2) - c \leq d_X( g\circ f(x_1), g\circ f(x_2)) \leq kd_X(x_1, x_2) + c,$$ 
and for all $y_1, y_2 \in Y$,
$$\frac{1}{k}d_Y(y_1, y_2) - c \leq d_Y( f\circ g(y_1), f\circ g(y_2)) \leq kd_Y(y_1, y_2) +c.$$
\end{itemize}
\end{definition} 

\begin{example} $\R$ is $(1,1)$--quasi-isometric to $\Z$.  Consider $f: \R \to \Z$, defined by $f(x) = \lfloor x \rfloor$, the floor function. Then for all $x, y \in \R$, $|x - y| - 1 \leq |\lfloor x \rfloor - \lfloor y\rfloor| \leq |x - y| + 1$. This map is clearly surjective.

Further, the inclusion $g: \Z \hookrightarrow \R$ is a quasi-inverse: For any $n \in \Z$, $f \circ g(n) = n$, and for any $x \in \R$, $g \circ f(x) = \lfloor x \rfloor$. So, if $m, n \in \Z$,
$$|m - n| = |f \circ g(m) - f \circ g(n)| = |m - n|,$$
and if $x, y \in \R$,
$$|x - y| - 1 \leq |g \circ f(x) - g \circ f(y)| \leq |x - y| + 1.$$
\end{example}

\begin{example}
\label{example:R2qiZ2}
$\R^2$ is $(2,2)$--quasi-isometric to $\Z^2$. We will go through the calculation, but the idea is simple:  rounding points in the plane down 
to points in the integer grid never distorts distances by too much, even when you change from $\ell^2$ to $\ell^1$ distance.
Consider $f: \R^2 \to \Z^2$, defined by $f(x, y) = (\lfloor x \rfloor, \lfloor y \rfloor)$. Then for any $(a, b), (x, y) \in \R^2$,

$$
\begin{array}{lllr}
d_{\Z^2}(f(a, b), f(x, y)) & = & |\lfloor x \rfloor - \lfloor a \rfloor| + |\lfloor y \rfloor - \lfloor b \rfloor | & \\
\\[ -7pt]
 & \leq & (|x - a| + 1) + (|y - b| + 1) & (\mbox{as above})\\
\\[ -7pt]
 & \leq & 2\max \{ |x - a|, |y - b| \} + 2 & \\
\\[ -7pt]
 & \leq & 2 \sqrt{ (\max \{ |x - a|, |y - b| \})^2 } + 2 & \\
\\[ -7pt]
 & \leq & 2 \sqrt{ (x - a)^2 + (y - b)^2 } + 2 & \\
\\[ -7pt]
 & = & 2d_{\R^2}((a, b), (x, y)) + 2 & \\
\end{array}
$$

and

$$
\begin{array}{lllr}
d_{\Z^2}(f(a, b), f(x, y)) & = & |\lfloor x \rfloor - \lfloor a \rfloor| + |\lfloor y \rfloor - \lfloor b \rfloor | & \\
\\[ -7pt]
 & \geq & (|x - a| - 1) + (|y - b| - 1) & (\mbox{also as above})\\
\\[ -7pt]
 & \geq & d_{\R^2}((a, b), (x, y)) - 2 & (\mbox{by the Triangle Inequality})\\
\\[ -7pt]
 & \geq & \frac{1}{2} d_{\R^2}((a, b), (x, y)) - 2 & \\
\end{array}
$$

It is easy to see that the inclusion $g: \Z^2 \hookrightarrow \R^2$ is a quasi-inverse; the composition $g\circ f:\R^2\to\R^2$
moves points no more than $\sqrt 2$.

\begin{remark} Above, we used quasi-isometry constants $k = 2, c = 2$. It is a nice exercise to show that $k = \sqrt{2}, c = 2$ are actually the best constants possible. But often we will not care what the constants actually are -- only that they exist.
\end{remark}

\end{example}

\begin{definition} A \emph{quasi-geodesic} is a quasi-isometric embedding of the real line into a space. That is, a map $f: \R \to X$ such that for some $k \geq 1$, $c \geq 0$, for all $t_1, t_2 \in \R$, 
$$\frac{1}{k}|t_1 - t_2| - c \leq d_X(f(t_1), f(t_2)) \leq k|t_1 - t_2| + c.$$
\end{definition}

Quasi-geodesics are useful, for instance, in discrete spaces:  they can sit still for a bounded period of time, and can make jumps of bounded size, but in the large-scale they proceed with distance roughly equal to time elapsed.

\begin{example}

\label{example:Qtheta}
Denote the real ray in $\R^2$ from the origin that makes an angle of $\theta$ with the positive $x$-axis as $r_{\theta}$. Then we can consider the image of $r_{\theta}$ under the the quasi-isometry $f: \R^2 \to \Z^2$ from Example \ref{example:R2qiZ2}. The result is a (disconnected) quasi-geodesic in $Cay(\Z^2, \{(1, 0), (0, 1)\}$. 

In this case, if we connect successive lattice points of $f\circ r_\theta$ with geodesic segments, the result is a geodesic ray in $Cay(\Z^2, \{(1,0), (0, 1)\})$. Call this ray $Q_{\theta}$ (see Figure \ref{figure:Q}).

\begin{figure}[tb]

\begin{center}
\includegraphics[width=2in]{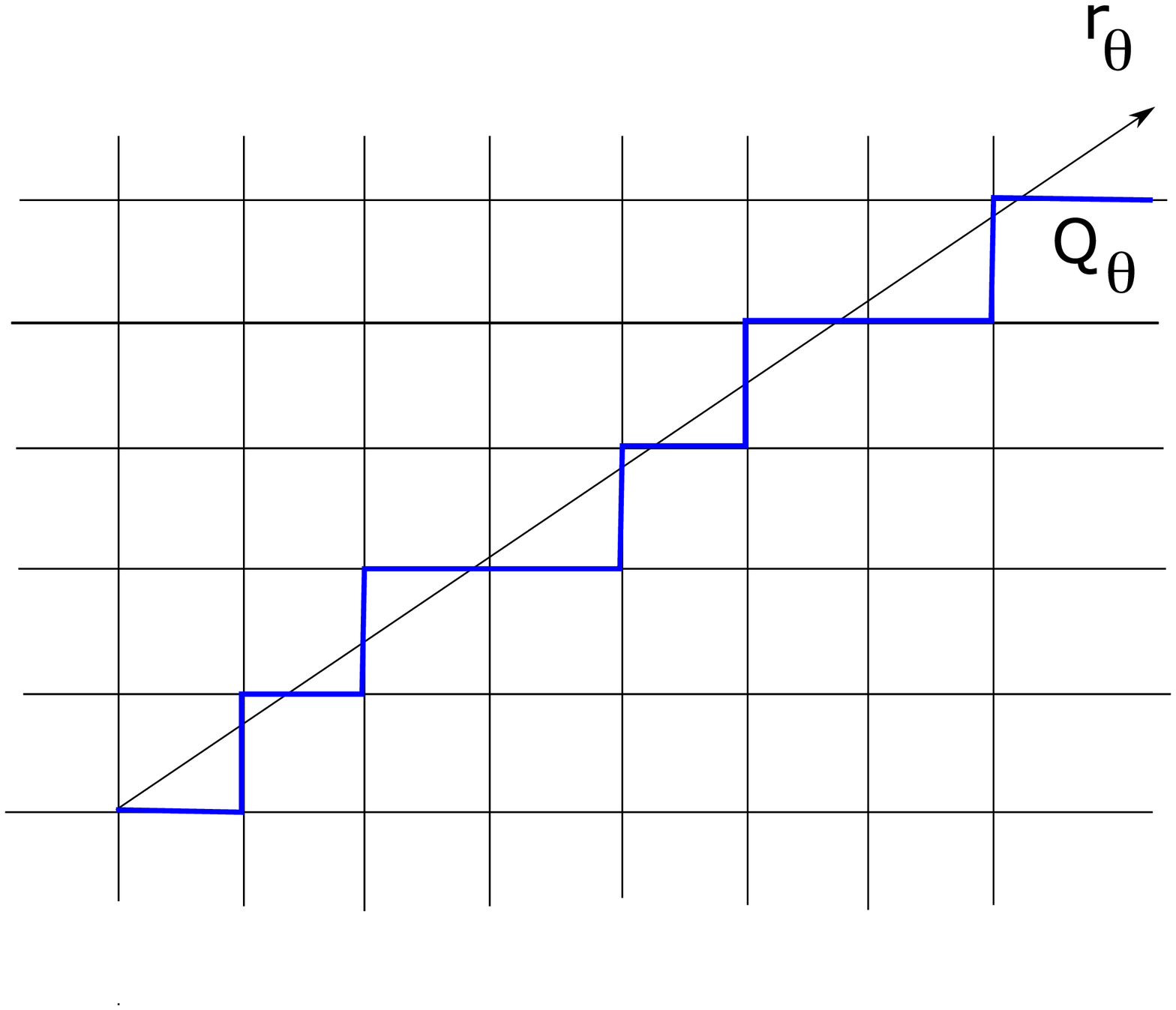}
\caption{$r_{\theta} \mapsto Q_{\theta}$ \label{figure:Q}
}
\end{center}
\end{figure}

\end{example}

Next, as promised, we confirm that the word metric is independent of the choice of generating set, up to quasi-isometry.

\begin{proposition} 
If $G$ is a finitely generated group with two (finite) generating sets $S$ and $S'$, then $Cay(G, S)$ is quasi-isometric to $Cay(G, S')$.
\end{proposition}

\begin{proof}
The identity map will be shown to be a quasi-isometry. Say $|S|=k$,  $|S'|=l$, let $d_S$ be the distance function in $Cay(G, S)$, and $d_{S'}$ in $Cay(G, S')$. Then, since $S$ and $S'$ are finite, let $m = \max \{ d_{S'}(s, e) \, | \, s \in S \}$, and $n= \max \{ d_{S}(s', e) \, | \, s' \in S' \}$, 
where $e\in G$ is the identity element.

Then, every element $g \in G$ can be written as a word in $S'$. And each of those generators can be written as words of $S$, each of length at most $n$. So $d_S(g, e) \leq n\cdot d_{S'}(g, e)$. To get the second inequality, the argument is reversed: $d_{S'}(g, e) \leq m\cdot d_S(g, e).$ So letting $k = \max \{m, n\}$ yields the quasi-isometry inequality. 

The argument is completed by noting that for any $g, h \in G$, $d(g, h) = d(h^{-1}g, e).$
\end{proof}

With this, we can speak unambiguously about the large-scale geometry of groups -- those properties of  groups that are invariant under quasi-isometry.

\subsection{The visual boundary}

\begin{notation}  Let $X$ be a geodesic space.  Then for $x_0 \in X,$ we define $$\G_{x_0}(X) = \{\hbox{unit speed geodesic rays from}~ x_0\}.$$ 
We will suppress $X$ from the notation and simply write $\G_{x_0}$.
\end{notation}

We want to think of light traveling along geodesics in the space $X$. So we think of the visual boundary as the set of all points one can ``see" at infinity, standing at the point $x_0$.

We give $\G_{x_0}$ the topology of uniform convergence on compact sets. Recall:

\begin{definition}
Let $(X,d)$ be a metric space and $Y$ a topological space. Given a
fixed element $f \in X^Y = \{$functions $g:\ Y \to X\}$, a
compact set $K$ of $Y$ and a number $\epsilon>0$, we let
$$B_K(f,\epsilon)\ =\ \{g \in X^Y\ |\ d(f(y),g(y))<\epsilon, \  \forall\  y \in K\}.$$

The sets $B_K(f,\epsilon)$ form a basis for the \emph{topology of uniform convergence on compact sets} on $X^Y$. 
\end{definition}

So  $\G_{x_0} \subset X^{\R}$, inherits the subspace topology. Roughly, if the images of two rays are ``close" on large compact sets, then the rays are ``close". And a sequence of rays converges to a limiting ray if the rays of the sequence agree with the limit on larger and larger compact sets. 

Sometimes, however, if we ``look" in different directions, we see the same point at infinity. To make this precise:

\begin{definition} We say that two geodesic rays, $g$ and $f$, are \emph{asymptotic} if there exists an $M \in \R$ such that $d(f(t), g(t)) \leq M$ for all $t$. This is an equivalence relation on rays. We will write $f \sim g$, and  denote the equivalence class of a ray $f \in \G_{x_0}$ by $[f]$, so $[f] = \{g \ | \ f \sim g\}$.
\end{definition}

\begin{definition}
The \emph{visual boundary} of a geodesic space $X$ at a point $x_0$,
denoted $\binf(X, x_0)$, is defined to be $\G_{x_0}/ \sim$, with the quotient topology. Let $\pi_{x_0} : \G_{x_0} \to \binf(X, x_0)$ be the natural projection map.
\end{definition}

\begin{example}The visual boundary of $\R^2$ at $(0, 0)$ is homeomorphic to the unit circle, $S^1$.
\end{example}

Again, the idea is simple: every geodesic ray from the origin corresponds to exactly one point on the unit circle, and exactly one point at infinity.  

\begin{proof}  Define a function $H : S^1 \to \binf(\R^2, (0,0))$ by $H(\theta) = \pi(r_{\theta})$, where $r_{\theta}$ is the straight line ray from the origin through the point on the unit circle corresponding to $\theta$ (see Figure \ref{figure:circle}).

\begin{figure}[tb]

\begin{center}
\includegraphics[width=1.5in]{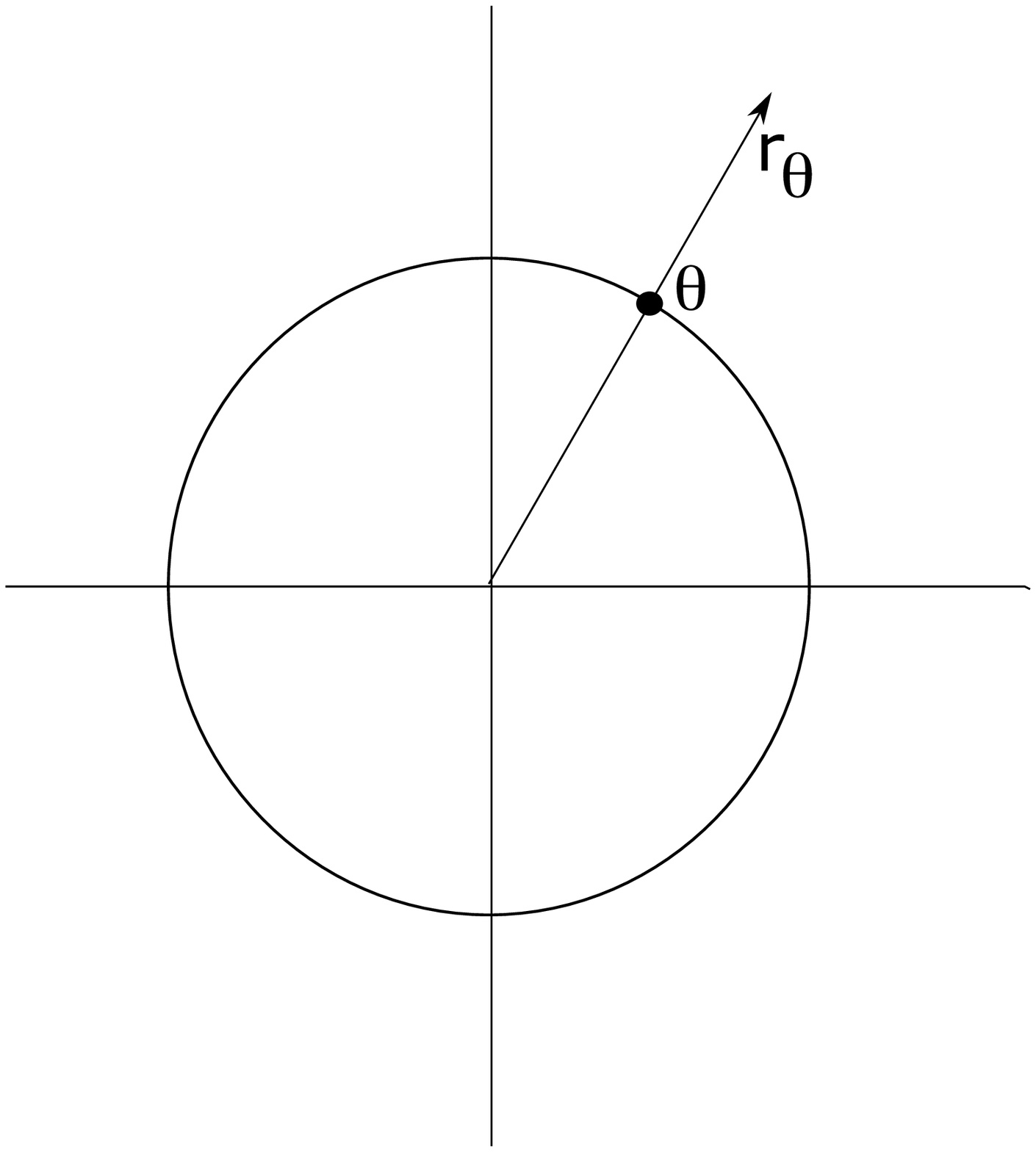}
\caption{$S^1 \to \binf(\R^2)$\label{figure:circle}
}
\end{center}
\end{figure}

To show that this map is a bijection, note that given any two distinct points on the circle, $\theta$ and $\phi$, the geodesic rays $r_{\theta}$ and $r_{\phi}$ diverge. That is, given any $M$, there exists some $T$ such that $d(r_{\theta}(t), r_{\phi}(t))>M$ for all $t > T$. Further, $H$ is clearly surjective, as the only geodesic rays in $\R^2$ are straight line rays.

To show that the map is continuous, we will examine open balls about arbitrary points.  Used implicitly in the remainder of the proof is the fact that $H$ and $\pi$ are bijections. 

Assume $V$ is open in $\binf(\R^2, (0,0))$.  Then $H^{-1}(V) = \{r(1) \ | \ r \in \pi^{-1}(V)\}$.  Now, consider an arbitrary point, $r^*(1) \in H^{-1}(V) \subset S^1$. We know what the basis of open sets in $\G$ looks like: it consists of the $B_K(f,\epsilon)$. So there exists an $\epsilon^*$ and a compact set $K = \{1\}$ such that the ball $B_{\{1\}}(r^*, \epsilon^*)$ is in $\pi^{-1}(V)$, because $\pi(r^*)$ is in $V$ and $\pi^{-1}(V)$ is open.  Then,	
$$H^{-1}(\pi(B_{\{1\}}(r^*, \epsilon^*))) = H^{-1}(\pi(\{r \ | \ d(r(t), r^*(t)) < \epsilon^*, t \in \{1\}\})$$

$$= H^{-1}(\pi(\{r \ | \ d(r(1), r^*(1)) < \epsilon^*\}) = \{r(1) \ | \ d(r(1),r^*(1)) < \epsilon^*\} $$

$$= B(r^*,\epsilon^*) \subset S^1$$

Now, assume $W$ is open in $S^1$.  We want to show that $H(W)$ is open.  Consider any ray $r^*$ such that $\pi(r^*) \in H(W)$.  Then we know there exists an
$\epsilon^*$ such that $B(r^*(1), \epsilon^*) = \{r(1) \ |\  d(r(1), r^*(1)) < \epsilon^*\} \subset W$.  Then, 
$$H(B(r^*(1), \epsilon^*)) = \{\pi(r) \ |\  d(r(1), r^*(1)) < \epsilon^*\}$$
$$= \{\pi(r) \ | \ d(r(t), r^*(t)) < \epsilon^*, t \in \{1\}\} = \pi(B_{\{1\}}(r^*, \epsilon^*))$$
Since, in this case, $\pi^{-1}(\pi(B_{\{1\}}(r^*, \epsilon^*))) = B_{\{1\}}(r^*, \epsilon^*)$ is open, so is its image.  Thus, given any point $\pi(r^*)$ in $H(W)$, there is an open set around this point contained in $H(W)$. We conclude that $H(W)$ is open, and ultimately that $H$ is a homeomorphism between $\binf(\R^2, (0,0))$ and $S^1$.

\end{proof}

We would like a way to talk about the visual boundary of a metric space, without reference to a specified basepoint. Unfortunately, there are many cases when the visual boundary changes if we use a different basepoint.

\begin{example} Consider, once again the set $X$ obtained by rotating $R$ about the $z$-axis (see Figure \ref{figure:pencil}). If we take our basepoint to be $(0, 0, -5)$, then $R$ is a geodesic ray from the basepoint, as is any rotation of $R$ about the $z$-axis. So $\G_{(0, 0, -5)}$ is a circle's worth of rays. If we take our basepoint to be $(1, 0, 0)$, on the other hand, the only geodesic ray from the basepoint is the ray $\{ (x, y, z) \, | \, x = 1, y = 0, z \geq 0 \}$. So $\G_{(1, 0, 0)}$ consists of a single ray.

Notice, however, that all the rays in $\G_{(0, 0, -5)}$ are asymptotic, since the distance between any two is bounded by $\pi$ (in the path metric). So when we take the quotient, we get $\binf(X, (0, 0, -5)) \cong \binf(X, (1, 0, 0)) \cong \{\mbox{point}\}$. 
\end{example}

In the above example, $\G$ depended on choice of basepoint, but the topological space $\binf(X)$ did not. In some spaces, however, even the visual boundary will change with the basepoint.

\begin{figure}[tb]

\begin{center}
\includegraphics[width=2in]{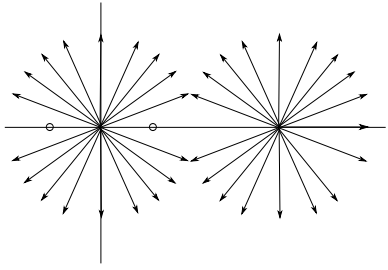}
\caption{The visual boundary of a twice-punctured plane depends on the choice of basepoint. \label{figure:2points}
}
\end{center}
\end{figure}

\begin{example}
Consider $X = \R^2 \setminus \{(1, 0), (-1, 0) \}$. Then if we choose the basepoint $(0, 0)$, there is a geodesic ray in every direction except along the positive and negative $x$-axes. So $\binf(X, (0, 0)) \cong (0, \pi) \cup (\pi, 2\pi)$. However, if we choose the basepoint $(3, 0)$, there is a geodesic ray in every direction except towards the negative $x$-axis. So $\binf(X, (2, 0)) \cong (0, 2\pi)$ (see Figure \ref{figure:2points}).
\end{example}

Fortunately, all is not lost.

\begin{proposition}  Given two points $x_1$ and $x_2$ in a geodesic space $X$, let $L:X \to X$ be an 
isometry carrying $x_1$ to $x_2$.  Then $\binf(X, x_1)$ is homeomorphic to $\binf(X, x_2)$.
 \label{proposition:basepoint}
\end{proposition}

\begin{proof} Isometries preserve geodesicity, so $\G_{x_1} \cong \G_{L(x_1)} = \G_{x_2}$. Further, the distance between geodesic rays is preserved, so $(\G_{x_1}/\sim) \cong (\G_{x_1}/\sim)$.
\end{proof}

\begin{remark}
When the isometry group of a space acts transitively on the space (e.g. $\R^2$ or $\Z^2$), we can suppress the basepoint. So we will denote $\binf(X,x_0)$ as simply $\binf(X)$, $\pi_{x_0}$ as $\pi$, and $\G_{x_0}$ as $\G$, when convenient.
\label{remark:basepoint}
\end{remark}

\begin{example}
In light of Remark \ref{remark:basepoint}, $\binf(\R^2) \cong S^1$.
\end{example}

\section{The case of $\Z^2$}

\subsection{Geodesic rays}

We will henceforth abuse notation, and identify $\Z^2$ with its Cayley graph with respect to the standard generating set, $Cay(\Z^2, \{(1, 0), (0, 1) \})$, the integer grid (Example \ref{example:Z2}). We will also implicitly assume the basepoint to be $(0,0)$. Geodesic paths consist of horizontal and vertical segments with no ``backtracking." 
As noted above, geodesics are not unique. For example, there are twenty geodesic paths between $(0,0)$ and $(3,3)$, all of length $6$ (see Figure \ref{figure:Z2}).

It is clear, then, that for any ray $f$, the equivalence class $[f]$ is ``large":  there are many geodesics $g$ such that $d(f,g)<M$ for all $t$.

\begin{notation}  An infinite ray in $\Z^2$, consisting of vertical and horizontal segments, can be expressed as an infinite string of the digits corresponding to each segment. Let 0, 1, 2, 3 and 4 represent east, north, west, south, and east respectively. Then any infinite ray in $\Z^2$ can be written as an infinite string over the alphabet $\{0,1,2,3,4\}$. (The redundant use of $0$ and $4$ for the eastward direction is to simplify later notation.)

If a ray is in the first quadrant, it can be written as a string over $\{0,1\}$; in the second, $\{1,2\}$; in the third, $\{2,3\}$; and in the fourth, $\{3,4\}$. To eliminate the only ambiguity, we adopt the convention that the east-pointing ray will be represented as the string $(\bar{0}) = (0, 0, 0, \dots)$ of all zeros.   
Given a geodesic ray $f \in \Z^2$, we will denote this expansion by $f = (f_1, f_2, f_3, ... )$.
Then if $m(f)=\min_n\{f_n\}$, we have $f_n\in\{m,m+1\}$ for all $n$.   
\end{notation}

\subsection{The topology on $\binf(\Z^2)$}

\begin{definition} We will say a ray $f$ in $\Z^2$ has \emph{slope $\theta$} if $f\sim Q_\theta$, where $Q_\theta$ is 
the ray  in direction $\theta$ (see Example \ref{example:R2qiZ2}).
\end{definition}

Note that not every ray has a slope. However, a ray cannot have more than one slope, because $\sim$ is transitive.

This sets us up to show that the visual boundary of $\Z^2$ is uncountable.

\begin{proposition} $|\binf(\Z^2)|=\mathfrak{c}$, the cardinality of the continuum.
\end{proposition}

In order to prove this, we will describe an injection from $S^1$ into $\binf(\Z^2)$, and an injection from $\binf(\Z^2)$ into $[0, 4)$, making use of the quinary expansions described in the previous section. 

\begin{proof}  The proof will proceed in two parts, exhibiting the two injections.

First, define the map $I: S^1 \to  \binf(\Z^2)$ to be given by $I(\theta) =
\pi(Q_{\theta})$, where $Q_{\theta}$ is the quasi-isometric embedding of $r_{\theta}$, the ray
that passes through the point $\theta$ on the unit circle in $\R^2$. Then
given any distinct $\theta, \phi \in S^1$, we have already seen that $\pi(Q_{\theta}) \neq \pi(Q_{\phi})$.
Thus $I$ is an injection, and $\mathfrak{c} \leq |\binf(\Z^2)|$.

For the second injection, recall that any geodesic ray can travel in at most two directions.  Hence, each ray corresponds to an infinite binary expansion.
Let these binary strings be mapped to the interval $[0,4)$ in the following way:

Let $B: {\{0, 1\} }^\N \to [0, 1]$ be the standard map from a binary expansion to the real number it represents. So $B((\epsilon_1, \epsilon_2, \epsilon_3, \dots)) = \displaystyle\sum_{n=1}^{\infty}\frac{\epsilon_n}{2^{n}}$, where $\epsilon_n \in \{0, 1\}$ for all $n$.

Now, for a quinary expansion, $( f_1, f_2, f_3, \dots ) \in {\{0, 1, 2, 3, 4\}}^\N$, let $m(f) = \min_n \{f_n\}$ as before. Then define a map $N: \G(\Z^2) \to [0, 4)$
by $$N((f_1, f_2, f_3, \dots )) = m + B((f_1 - m, f_2 - m, f_3 - m, \dots )).$$
So for instance, $N(\bar{0})=0$ and 
$$N((2, 3, 2, 3, 2, 3, \dots)) = 2 + B((0, 1, 0, 1, 0, 1, \dots )) = 2 + \frac{1}{3} = \frac{7}{3}.$$

It is easy to see (by uniqueness of binary expansions for the fractional part) that this map is injective from $\G(\Z^2) \to [0, 4)$. Thus $|\G(\Z^2)| \leq \mathfrak{c}$, so $|\binf(\Z^2)| \leq \mathfrak{c}$.

It follows that $|\binf(\Z^2)| = \mathfrak{c}$.

\end{proof}

\begin{proposition}  $\binf(\Z^2)$ possesses the trivial topology.
\end{proposition}

In other words, the only open sets in the visual boundary are the entire set and the empty set.

\begin{proof} By the quotient topology on $\G / \sim$, a set $U \subset \binf(\Z^2)$ is open exactly when its pre-image $\pi^{-1}(U)$ is open in $\G$.  Assume that $U$ is some
non-empty open set in $\binf(\Z^2)$. Then, $W=\pi^{-1}(U)$ is also open and non-empty.
We wish to show that $U$ is the entire set.  It suffices to show that given any $g \in \G$, $\pi(g) \in U$.

As $W$ is open and non-empty, there is some geodesic ray $f$ in $W$.
Consider any ray $g$ such that $m(g) = m(f)$. (This means that $g$ and $f$ are in the same quadrant.) We will show that given any compact set $K \subset [0, \infty)$ and $\epsilon > 0$, $g$ has
some representative $g_s \in [g]$ such that $g_s \in B_K(f,\epsilon) \subset W$. It will follow that $\pi(g) \in U$.

Let the compact set $K=[a,b]$ and $\epsilon>0$ be given, and let $s = \lceil b \rceil \in \Z$.
Then define the representative $g_s$ of $g$ as follows:

$$ {g_s(t)} = \left\{ \begin{array}{ccr} {f(t)} & \mbox{for} & t \leq  s \\ f( s )+g(t)-g( s ) & \mbox{for} &  t > s \end{array} \right.$$

where the sum is group addition on $\Z^2$.

\begin{figure}[tb]

\begin{center}
\includegraphics[width=2.5in]{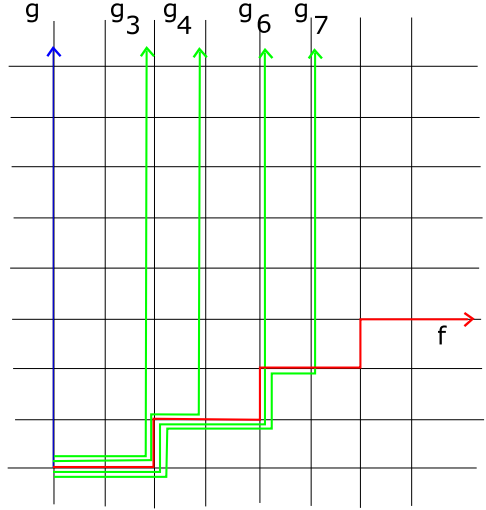}
\caption{A sequence of rays asymptotic to $g$, in a neighborhood of $f$
\label{figure:gb}
}
\end{center}
\end{figure}

To clarify, consider the infinite binary expansion of $g_s$.  This is identical to that of $f$ for the first $s$ steps (so $d(g_s(t), f(t)) = 0$ for $t \leq s$), and afterwards is identical to that of $g$ (so $g_s \sim g$)
(see Figure \ref{figure:gb}).

Clearly, $g_s \in B_K(f,\epsilon)$, so $\pi(g_s) \in \pi(B_K(f,\epsilon))$. Then since $B_K(f,\epsilon) \subset W$, $\pi(g_s) \in U$.  Finally, since $g_s \sim g$, $\pi(g_s)=\pi(g)$. We conclude that $\pi(g) \in U$ and $g \in W$. 

Recall that $f$ and $g$ are in the same quadrant because $m(f) = m(g) = m$.  In particular, we see that the axis geodesic $h = (\overline{m+1}) \in B_K(f,\epsilon)$, where we take addition $mod \ 4$.

We now take advantage of the fact that $W$ is open. If $h \in W$ then there must exist some $\epsilon'$ such that $B_K(h,\epsilon') \subset W$. Then let $j = (\overline{m+2})$ and let $j_s$ be the representative function as above, so that $j_s \in B_K(h, \epsilon')$. Therefore $\pi(j) \in U$. Consequently, all axis directions are in $U$.  By the same argument, then, we include in the set $U$ the images of all other non-axis geodesic rays $g$ for which $m(g) \not = m(f)$.  We can then conclude that given any geodesic ray $g \in \Z^2$, $\pi(g) \in U$.

By assuming only that $U$ was open and non-empty, we showed that $U$ contains all elements of $\binf(\Z^2)$.
We conclude that $\binf(\Z^2)$ has the trivial topology.

\end{proof}

\section{Further Comments}
Informally speaking, if we ``zoom out" from $\Z^2$ by rescaling distances to be smaller and smaller, we limit to $\R^2$ with the $\ell^1$-norm. (Formally, this construction is called the {\em asymptotic cone}, and $\mathrm{Cone}(\Z^2)=(\R^2,  \ell^1$).) 
We expect the same method of proof from above to show that the visual boundary of $(\R^2, \ell^1)$ is an uncountable set with the trivial topology. And in fact, this is true.

\begin{proposition} $\binf((\R^2, \ell^1))$ has the cardinality $\mathfrak{c}$, and the trivial topology.
\end{proposition}

\begin{proof}
Geodesic rays are no longer restricted to vertical and horizontal segments, but they have a similar property. Let us first discuss geodesic rays that enter the interior of the first quadrant. Let $f(t) = (f_1(t), f_2(t))$ be a geodesic ray from the origin, passing through the point $f(t_0) = (x, y)$, with $x, y > 0$. Then for all $t > t_0$, $f_1(t) \geq x$, and $f_2(t) \geq y$. In other words, once a geodesic begins to move in a north-westerly direction, it can never again move toward the south or east (see Figure \ref{figure:ell1}). A similar property, of course, also holds in the other quadrants.  

\begin{figure}[tb]

\begin{center}
\includegraphics[width=2.5in]{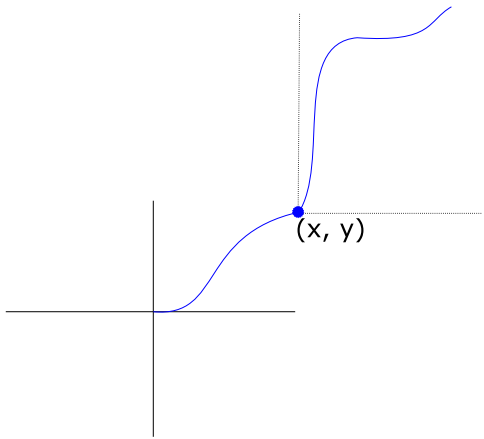}
\caption{A geodesic in $(\R^2, \ell^1)$.
\label{figure:ell1}
}
\end{center}
\end{figure}

There are more geodesic rays in this space than in $\Z^2$. But after we take the quotient, we get the same boundary. We will appeal to the Axiom of Choice. Certainly, $|\binf((\R^2, \ell^1))|$ is at least $\mathfrak{c}$, since each geodesic ray in $\Z^2$ includes as a geodesic ray into $(\R^2, \ell^1)$. Now, for each equivalence class of asymptotic rays $[f] \in \binf((\R^2, \ell^1))$, choose a representative geodesic ray, $f$. Then, as before, consider the image of this ray under the quasi-isometry from $\R^2$ onto $\Z^2$, and connect vertices by horizontal and vertical segments to get $Q$, a geodesic ray in $\Z^2$. Identifying $Q_f$ and $Q_g$ with their images in $\R^2$ by inclusion, we see that $f\sim Q_f$ and $g\sim Q_g$, 
so the map from $\binf((\R^2, \ell^1))$ to $\binf(\Z^2)$ is an injection. This establishes that $|\binf((\R^2, \ell^1))| = \mathfrak{c}$. 

Next, we use an identical construction to the one above to show that the topology is trivial.

Let $f, g$ be any arbitrary geodesic rays in the closure of quadrant $I$.  We will show that given any compact $K \subset [0, \infty)$ and $\epsilon > 0$, $g$ has a representative $g_b \in [g]$ such that $g_b \in B_K(f,\epsilon)$. 

Let the compact set be $K=[a,b]$ and $\epsilon>0$ be given.
Then define the representative $g_b$ of $g$ as follows:

$$ {g_b(t)} = \left\{ \begin{array}{ccr} {f(t)} & \mbox{for} & t \leq  b \\ {f(b)+g(t)-g(b)} & \mbox{for} & t > b \end{array} \right.$$

where now the sum is component addition on $\R^2$.

Just as before, this argument establishes that any open set containing a single ray in quadrant $I$ contains all rays in quadrant $I$, and can be extended to show that any non-empty open set contains every ray.
\end{proof}

What's wrong with this state of affairs?  This boundary completely fails to be Hausdorff:  we can't separate any two directions at infinity.  
Convergence to a particular point in the boundary is meaningless.

To see some of the consequences of this finding, consider  that a large class of groups have an undistorted free abelian subgroup; that is, a quasi-isometric embedding of $\Z^2 \cong \langle a, b \rangle$,  called a \emph{quasi-flat}.  These arise whenever two elements commute and there is no ``shortcut" to words in those elements coming from a relator.  Besides the obvious extension of the same argument to $\Z^n$, quasi-flats can also be found in right-angled Artin groups, as well as mapping class groups of surfaces.
Papasoglu shows  in \cite{PP} that every semi-hyperbolic group which is not hyperbolic contains such a quasi-flat. This includes fundamental groups of compact manifolds of non-positive curvature. 
So this note shows, in particular, that any metric space containing a quasi-flat will have a bad boundary.

\end{document}